\documentclass[a4paper,11pt]{article}
\usepackage[latin1]{inputenc}
\usepackage[english]{babel}
\usepackage{amsmath}
\usepackage{amsfonts}
\usepackage{amssymb}
\usepackage{epsfig}
\usepackage{amsopn}
\usepackage{amsthm}
\usepackage{color}
\usepackage{graphicx}
\usepackage{subfigure}
\usepackage{enumerate}
\usepackage{tikz-cd}
\newtheorem{theorem}{Theorem}[section]

\newtheorem{proposition}[theorem]{Proposition}
\newtheorem{definition}[theorem]{Definition}

\newtheorem*{theorem*}{Theorem}
\newtheorem*{lemma*}{Lemma}
\newtheorem*{remark*}{Remark}
\newtheorem*{definition*}{Definition}
\newtheorem*{proposition*}{Proposition}
\newtheorem*{corollary*}{Corollary}
\numberwithin{equation}{section}
\newcommand{\vertiii}[1]{{\left\vert\kern-0.25ex\left\vert\kern-0.25ex\left\vert #1
    \right\vert\kern-0.25ex\right\vert\kern-0.25ex\right\vert}}

\newcommand{\real}{\mathbb{R}}

\let\ced=\c         

\def\qed{\,\unskip\kern 6pt \penalty 500
\raise -2pt\hbox{\vrule \vbox to8pt{\hrule width 6pt
\vfill\hrule}\vrule}\par}
\definecolor{darkblue}{rgb}{0.05, .05, .65}
\definecolor{darkgreen}{rgb}{0.1, .65, .1}
\definecolor{darkred}{rgb}{0.8,0,0}
\newcommand{\beqn}{\begin{equation}}
\newcommand{\eeqn}{\end{equation}}
\newcommand{\bear}{\begin{eqnarray}}
\newcommand{\eear}{\end{eqnarray}}
\newcommand{\bean}{\begin{eqnarray*}}
\newcommand{\eean}{\end{eqnarray*}}
\begin{document}


\title{\huge \bf Blow-up rates and sets for a quasilinear diffusion equation with weighted source}

\author{\Large Ra\'ul Ferreira\,\footnote{Departamento de An\'alisis Matem\'atico y Matem\'{a}tica
Aplicada, Universidad Complutense de Madrid, 28040, Madrid, Spain, \textit{e-mail:} raul\textunderscore ferreira@mat.ucm.es},
\\[4pt]
\Large Razvan Gabriel Iagar\,\footnote{Departamento de Matem\'{a}tica
Aplicada, Ciencia e Ingenieria de Materiales y Tecnologia
Electr\'onica, Universidad Rey Juan Carlos, M\'{o}stoles,
28933, Madrid, Spain, \textit{e-mail:} razvan.iagar@urjc.es. Corresponding Author.},
\\[4pt] \Large Ariel S\'{a}nchez,\footnote{Departamento de Matem\'{a}tica
Aplicada, Ciencia e Ingenieria de Materiales y Tecnologia
Electr\'onica, Universidad Rey Juan Carlos, M\'{o}stoles,
28933, Madrid, Spain, \textit{e-mail:} ariel.sanchez@urjc.es}\\
[4pt] }
\date{}
\maketitle

\begin{abstract}
Blow-up rates are established for general solutions to the quasilinear diffusion equation
$$
\partial_tu=\Delta u^m+|x|^{\sigma}u^p, \quad (x,t)\in\real^N\times(0,T),
$$
in the range of exponents $1<p<m$, $\sigma>0$. More precisely, if we consider a compactly supported solution $u(x,t)$ with blow-up time $T=T(u)\in(0,\infty)$, we derive the blow-up rate
$$
C_1(T-t)^{-\alpha}\leq \|u(x,t)\|_{\infty}\leq C_2(T-t)^{-\alpha}, \quad t\in(0,T),
$$
for some positive constants $C_1$, $C_2$, and the upper rate of expansion of the support
$$
\sup\{|x|:u(x,t)>0\}\leq C_0(T-t)^{-\beta}, \quad t\in(0,T),
$$
for some constant $C_0>0$, where
$$
\alpha=\frac{\sigma+2}{L}, \quad \beta=\frac{m-p}{L}, \quad L=\sigma(m-1)+2(p-1).
$$
We also analyze the blow-up sets of solutions $u$, showing, under a suitable condition, that either $B(u)=\real^N$ or blow-up takes place only as $|x|\to\infty$.
\end{abstract}

\


\smallskip


\section{Introduction}

The reaction-diffusion equation
\begin{equation}\label{eq1}
\partial_tu=\Delta u^m+|x|^{\sigma}u^p, \quad (x,t)\in Q_T:=\real^N\times(0,T), \quad T>0,
\end{equation}
presents an interesting competition between the effect of the diffusion and the one of the source term. The former is a conservative process for $m>1$, preserving the total mass while expanding the support of any compactly supported solution, while the latter produces an increase of the $L^1$ norm of any solution, leading eventually to a finite time blow-up. By finite time blow-up of a solution $u$ to Eq. \eqref{eq1} we understand the existence of a time $T(u)\in(0,\infty)$, called \emph{the blow-up time} of $u$, such that $u(t)\in L^{\infty}(\real^N)$ for any $t\in(0,T(u))$, but $u(T(u))\notin L^{\infty}(\real^N)$. Throughout the paper, we will drop from the notation of the blow-up time the dependence on $u$ when there is no danger of confusion. In the case of Eq. \eqref{eq1}, the competition described above is furthermore influenced by the variable coefficient $|x|^{\sigma}$, which enhances the effect of the reaction for large values of $|x|$ while reducing it in a neighborhood of the origin, introducing some noticeable differences in the behavior when $\sigma>0$ is large. We restrict ourselves to the range of exponents
\begin{equation}\label{range.exp}
1<p<m, \quad \sigma>0,
\end{equation}
and consider initial conditions satisfying
\begin{equation}\label{icond}
u(x,0)=u_0(x), \quad u_0\in C_{0,r}(\real^N), \quad u_0(x)\geq 0 \quad {\rm for \ any} \ x\in\real^N, \quad u_0\not\equiv0,
\end{equation}
where $C_{0,r}(\real^N)$ means the class of continuous, radially symmetric and compactly supported functions in $\real^N$. Properties such as well-posedness, comparison principle, finite speed of propagation and finite time blow-up of any solution of the Cauchy problem \eqref{eq1}-\eqref{icond} have been proved in the paper \cite{ILS24b}.

The aim of this paper is thus to shed some light on the finer blow-up behavior of solutions to the Cauchy problem \eqref{eq1}-\eqref{icond}, in the range of exponents \eqref{range.exp}, trying to answer to the following important questions related to the properties of the solutions near their blow-up time:

$\bullet$ deducing blow-up rates, that is, the time scale of $\|u(t)\|_{\infty}$ as $t\to T$.

$\bullet$ describing the blow-up set $B(u_0)$, that is, the set of points $x\in\real^N$ where the solution $u(x,t)$ to the Cauchy problem \eqref{eq1}-\eqref{icond} is unbounded as $t\to T$.

The previous questions are nowadays answered for the non-weighted reaction-diffusion equation (corresponding to taking $\sigma=0$ in Eq. \eqref{eq1}), and two rather complete monographs devoted to this case (containing much more information and focusing also on larger values of $p$) are available, \cite{QS} for the semilinear range $m=1$ and \cite{S4} for the quasilinear range $m>1$. In particular, letting $\sigma=0$ in Eq. \eqref{eq1}, it has been proved that, for initial conditions as in \eqref{icond}, the blow-up behavior is self-similar (see for example \cite[Chapter IV, Sections 5-6]{S4} or \cite{G95}). However, a number of techniques employed in the above mentioned references cannot be easily extended to Eq. \eqref{eq1} with $\sigma>0$ due to difficulties stemming from the unboundedness of the variable coefficient.

The influence of a space-dependent coefficient of the source term over the qualitative and dynamic properties of solutions to reaction-diffusion equations has been an interesting subject to study starting from papers such as \cite{BK87, Pi97, Pi98}, where local existence, global existence and finite time blow-up are discussed in the semilinear case $m=1$ and with rather general weights. Still in the semilinear case, the question of the blow-up set (and, noticeably, whether $x=0$ can or cannot be a blow-up point to Eq. \eqref{eq1} with $m=1$ and $\sigma>0$) has been discussed in a series of works \cite{GLS, GS11, GLS13, GS18}. The question related to how the blow-up takes place, in the sense of blow-up rates and self-similar profiles, has been addressed (still with $m=1$, $p>1$) in works such as \cite{FT00, MS21}, self-similar solutions being constructed in the former, while different blow-up rates than the expected self-similar ones are obtained in the latter for $p$ large.

Let us now discuss the state of art of the problem in our case of interest, which is Eq. \eqref{eq1} with $m>1$ and $\sigma>0$. The well-posedness of the Cauchy problem, in a very general framework, follows from Andreucci and DiBenedetto \cite{AdB91}. The question  of classifying initial conditions $u_0$ as in \eqref{icond} for which finite time blow-up occurs has been raised and partially answered in \cite{Qi98, Su02}, where a Fujita-type exponent $p_F=m+(\sigma+2)/N>m$ and a second critical exponent related to the decay of $u_0$ as $|x|\to\infty$ ensuring global existence when $p>p_F$ have been identified. Under some technical limitations on $\sigma$, but in the even more general case of considering the doubly nonlinear diffusion operator, blow-up rates were obtained in \cite{AT05} for $p>m$, proving that they coincide with the expected self-similar ones. Recently, directions for blow-up at infinity of suitable solutions have been obtained in \cite{SU21, SU24} (as well with $p>m$). A lot of progress has been done in understanding the dynamic properties of solutions to the equation
$$
u_t=\Delta u^m+a(x)u^p,
$$
where $a(x)$ is a continuous and compactly supported weight, see for example \cite{FdPV06, FdP18, FdP22, KWZ11, Liang12}. 

However, most of the previously mentioned results have been obtained either when $p>m$ or when the variable coefficient is bounded and compactly supported. Motivated by this lack of analysis of Eq. \eqref{eq1} in the range of exponents \eqref{range.exp}, two of the authors started a larger project of understanding the effect of the presence of an unbounded weight on the reaction term (taking $|x|^{\sigma}$ in order to preserve a scaling invariance and thus a self-similar structure). Our approach started from classifying the self-similar profiles to Eq. \eqref{eq1}, and restricting ourselves to the range \eqref{range.exp}, it was shown in \cite{IS21, ILS24a} that self-similar profiles presenting finite time blow-up always exist, but their blow-up sets strongly depend on the magnitude of $\sigma$: for $\sigma>0$ sufficiently small blow-up occurs in the whole space $\real^N$, while for $\sigma>0$ larger, all the self-similar solutions blow up only at infinity, a phenomenon analyzed previously in \cite{La84, GU06} for $\sigma=0$ and \cite{SU21} for more general space-dependent coefficients instead of $|x|^{\sigma}$, but only for $p>m$ and strictly positive initial conditions. We will state rigorously what blow-up at infinity means, in the section dedicated to the main results. With respect to self-similar solutions, the limiting cases $p=1$, respectively $p=m$ have been also considered in \cite{IS19}, respectively \cite{IS20, IS22, IL22}, showing that at least in the latter one, the classification of self-similar profiles strongly departs from the range $1<p<m$.

The next and natural step in the analysis, once having all the information on self-similar solutions, is to derive qualitative properties of general solutions to the Cauchy problem \eqref{eq1}-\eqref{icond}. A first step in this direction, in the range \eqref{range.exp}, has been performed in the recent paper \cite{ILS24b}. It has been proved therein that the Cauchy problem \eqref{eq1}-\eqref{icond} is well-posed, the comparison principle holds true and any non-trivial solution to it blows up in a finite time $T=T(u)$. Moreover, concerning the expansion of the support of a solution $u$, finite speed of propagation is established for $t\in(0,T(u))$, while the support becomes infinite (a phenomenon called \emph{blow-up of the interface}) as $t\to T(u)$.

The present work aims at going forward in the analysis of the Cauchy problem \eqref{eq1}-\eqref{icond} by deducing the blow-up rate of any solution $u$ and, when possible, analyzing its blow-up set when $t\to T(u)$. It is thus the right moment to introduce our main results.

\medskip

\noindent \textbf{Main results.} Let $m$, $p$ and $\sigma$ fulfill \eqref{range.exp}. Let $u_0$ be an initial condition as in \eqref{icond} and $u$ be the solution to the Cauchy problem \eqref{eq1}-\eqref{icond}, which is unique and blows up at a time $T\in(0,\infty)$, being also compactly supported for $t\in(0,T)$, as follows from \cite{ILS24b}. We set, for simplicity of the notation,
$$
Q_T:=\real^N\times(0,T).
$$
and write throughout the paper $u(t)$ for the mapping $x\mapsto u(x,t)$ for $t>0$ fixed. For the sake of completeness, we give below the definitions of the concepts we employ in the paper, starting with the notion of solution.
\begin{definition}[Weak solution]\label{def.weak}
We say that a function $u$ is a \emph{weak solution} to the Cauchy problem \eqref{eq1}-\eqref{icond} in $Q_T$ if $u\in C(\real^N\times(0,t))$ for any $t\in(0,T)$ and if, for any bounded domain $\Omega\subset\real^N$ and for any test function $\varphi\in C^{2,1}(\overline{\Omega}\times[0,T))$ such that $\varphi=0$ on $\partial\Omega\times(0,T)$, we have the following equality
\begin{equation}\label{weaksol}
\begin{split}
\int_{\Omega}u(x,t)\varphi(x,t)\,dx&-\int_{\Omega}u_0(x)\varphi(x,0)\,dx\\&=\int_0^t\int_{\Omega}\left(u(x,s)\partial_s\varphi(x,s)+u^m(x,s)\Delta\varphi(x,s)+|x|^{\sigma}u^p(x,s)\right)\,dx\,ds\\
&-\int_0^t\int_{\partial\Omega}u^m(x,s)\partial_{n}\varphi(x,s)\,dS\,ds,
\end{split}
\end{equation}
for any $t\in(0,T)$. We say that $u$ is a weak supersolution to the Cauchy problem \eqref{eq1}-\eqref{icond} if we replace equality by the sign $\geq$ in \eqref{weaksol} and we say that $u$ is a weak subsolution if we replace equality by the sign $\leq$ in \eqref{weaksol}, in both cases restricting the inequalities to non-negative test functions.
\end{definition}
We also define below what do we understand by a blow-up point and by blow-up at space infinity, following \cite{GU06}.
\begin{definition}[Blow-up point and blow-up set]\label{def.BUS}
Let $u$ be a solution to the Cauchy problem \eqref{eq1}-\eqref{icond} in $Q_T$, with $T\in(0,\infty)$ its blow-up time. A point $x\in\real^N$ is called a \emph{blow-up point} of $u$ if there exists a sequence $(x_k,t_k)_{k\geq1}$ in $Q_T$ such that
\begin{equation}\label{BUP}
\lim\limits_{k\to\infty}x_k=x, \quad \lim\limits_{k\to\infty}t_k=T, \quad \lim\limits_{k\to\infty}u(x_k,t_k)=\infty.
\end{equation}
We say that $u$ \emph{blows up at space infinity} if there exists a sequence $(x_k,t_k)_{k\geq1}$ in $Q_T$ such that
\begin{equation}\label{BUI}
\lim\limits_{k\to\infty}|x_k|=\infty, \quad \lim\limits_{k\to\infty}t_k=T, \quad \lim\limits_{k\to\infty}u(x_k,t_k)=\infty.
\end{equation}
The \emph{blow-up set} $B(u)\subseteq\real^N\cup\{\infty\}$ of a solution $u$ is the set containing its blow-up points. We say that $u$ blows up \emph{only at space infinity} if \eqref{BUI} holds true but \eqref{BUP} is false for any $x\in\real^N$.
\end{definition}
Since $u$ and $T(u)$ are uniquely determined by the initial condition $u_0$, we will sometimes write $B(u_0)$ instead of $B(u)$. We also write $B(u)=\{\infty\}$ to express the fact that the solution $u$ blows up only at infinity.

\medskip

Our first result concerns the blow-up rates of solutions to \eqref{eq1}-\eqref{icond} as $t\to T$. To this end, let us introduce the following exponents
\begin{equation}\label{SSexp}
\alpha=\frac{\sigma+2}{L}, \quad \beta=\frac{m-p}{L}, \quad L=\sigma(m-1)+2(p-1),
\end{equation}
which are exactly the self-similarity exponents of Eq. \eqref{eq1} in the range \eqref{range.exp}, according to \cite{ILS24a}. We then have the following result:
\begin{theorem}[Blow-up rates]\label{th.1}
Let $m$, $p$ and $\sigma$ be as in \eqref{range.exp} and let $u$ be a solution to the Cauchy problem \eqref{eq1}-\eqref{icond} and $T\in(0,\infty)$ its blow-up time. Then there exist positive constants $0<C_1<C_2<\infty$ such that
\begin{equation}\label{rates}
C_1(T-t)^{-\alpha}\leq \|u(t)\|_{\infty}\leq C_2(T-t)^{-\alpha},
\end{equation}
for any $t\in(0,T)$, where $\alpha$ is given in \eqref{SSexp}.
\end{theorem}
Let us observe that the blow-up rates depend on $\sigma>0$, enhancing the effect of the space-dependent coefficient, but not on the dimension $N$ of the space. Indeed, it is well known that, for $\sigma=0$, the blow-up behavior is self-similar and, in particular, the blow-up rate is $(T-t)^{-1/(p-1)}$, independent of $m>1$ (see for example \cite{G95}). We thus observe that the dependence with respect to $\sigma$ of the rates in \eqref{rates} matches this limiting case. Let us also notice that $\alpha\to1/(m-1)$ as $\sigma\to\infty$, which may suggest that a rate $(T-t)^{-1/(m-1)}$ is expected if we replace $|x|^{\sigma}$ by an exponential weight. The proof of Theorem \ref{th.1} is more involved than in the case $\sigma=0$ and strongly relies on the classification of self-similar solutions established in \cite{IS21} (dimension $N=1$) and \cite{ILS24a} (dimension $N\geq2$).

We delve next our attention to the blow-up points of solutions to the Cauchy problem \eqref{eq1}-\eqref{icond}. In the framework of compactly supported self-similar solutions in the form
\begin{equation}\label{SSS}
U(x,t;\sigma)=(T-t)^{-\alpha}f(\zeta), \quad \zeta:=|x|(T-t)^{\beta},
\end{equation}
with $\alpha$, $\beta$ given in \eqref{SSexp} and the profile $f$ being a solution to the differential equation
\begin{equation}\label{SSODE}
(f^m)''(\zeta)+\frac{N-1}{\zeta}(f^m)'(\zeta)-\alpha f(\zeta)+\beta\zeta f'(\zeta)+|\zeta|^{\sigma}f^p(\zeta)=0,
\end{equation}
for $\zeta\geq0$, a striking difference has been identified with respect to the non-weighted case $\sigma=0$. More precisely, it is established in \cite{IS21, ILS24a} that, according to the magnitude of $\sigma>0$, there are self-similar solutions in the form \eqref{SSS} exhibiting two different types of blow-up sets: either $B(u)=\real^N$, which occurs for $\sigma>0$ small, or $B(u)=\{\infty\}$ (in the sense of Definition \ref{def.BUS}), which occurs for $\sigma>0$ sufficiently large. We shall prove in Section \ref{sec.points} that, under some technical limitations (see \eqref{interm19}), the two previous blow-up sets are the only possible ones for any solution to Eq. \eqref{eq1}, see Propositions \ref{prop.bounded} and \ref{prop.unbounded} for precise statements. A discussion of cases when either one of these two possibilities is taken ends the paper. 

\medskip

\noindent \textbf{Open problems related to blow-up.} Three natural \emph{open questions} arise immediately from the previous discussion. First, one should aim at proving without any extra condition that $B(u_0)=\real^N$ and $B(u_0)=\{\infty\}$ are the only two possibilities, for whatever initial condition $u_0$. Second, another open question is to establish which one of these two possibilities is taken by a solution stemming from an initial condition such that $u_0(0)>0$. The main difficulty in solving the second question is that, for $\sigma$ large enough, there is a shortage of solutions (in self-similar form), as shown in \cite{IS21, ILS24a}, thus techniques based on intersection comparison are no longer available and new ideas are to be found. Finally, another interesting open problem is to establish the behavior of a solution $u$ to the Cauchy problem \eqref{eq1}-\eqref{icond} near the blow-up time. Of course, according to the previous discussion, it is expected that the solutions should tend towards self-similar solutions in the form \eqref{SSS} whose profiles $f$ are solutions to \eqref{SSODE}, but, up to our knowledge, proving such large time behavior is in general a very difficult problem.

\medskip

We are now in a position to give the proofs of our main results, and the remaining part of the paper is split into two sections.

\section{Blow-up rates}\label{sec.rates}

This section is dedicated to the proof of Theorem \ref{th.1}. It is easy to notice that, if $u_0$ is radially symmetric, then the solution $u$ to \eqref{eq1}-\eqref{icond} is radially symmetric for any $t\in(0,T)$, by the rotational invariance of Eq. \eqref{eq1}. Let us denote by
\begin{equation*}
R(t):=\sup\{|x|: x\in\real^N, u(x,t)>0\}, \quad t\in[0,T),
\end{equation*}
the edge of the support of $u$, where $T=T(u)\in(0,\infty)$ is its blow-up time. The following upper bound for $R(t)$ will be essential in the proof of the blow-up rates.
\begin{proposition}\label{prop.support}
Let $u$ be a solution to the Cauchy problem \eqref{eq1}-\eqref{icond} with blow-up time $T\in(0,\infty)$. Then, there exists $C_0>0$ depending only on the initial condition $u_0$ (and on $m$, $p$, $\sigma$, $N$) such that
\begin{equation}\label{uppersup}
R(t)\leq C_0(T-t)^{-\beta}, \quad t\in(0,T),
\end{equation}
with $\beta$ defined in \eqref{SSexp}.
\end{proposition}
\begin{proof}
Consider a radially symmetric subsolution in self-similar form
\begin{equation}\label{interm1}
\underline{U}(x,t)=(T-t)^{-\alpha}\underline{f}(|x|(T-t)^{\beta}),
\end{equation}
with $\alpha$, $\beta$ given by \eqref{SSexp} and a profile $\underline{f}$ such that
\begin{equation}\label{interm2}
\underline{f}(\zeta_0)=\underline{f}(\zeta^0)=0, \quad {\rm supp}\underline{f}=[\zeta_0,\zeta^0], \quad (\underline{f}^m)'(\zeta_0)>0, \quad
(\underline{f}^m)'(\zeta^0)=0.
\end{equation}
The existence of such profiles is ensured by \cite[Proposition 3.4]{IS21} in dimension $N=1$, respectively \cite[Steps 3-4, Section 4]{ILS24a} in dimension $N\geq2$, and the proofs therein entail furthermore that one can choose $\underline{f}$ and $\underline{U}$ as in \eqref{interm2} and \eqref{interm1} such that
\begin{equation}\label{interm3}
R(0)<\zeta_0T^{-\beta}<\zeta^0T^{-\beta}.
\end{equation}
Setting
$$
\zeta_0(t):=\zeta_0(T-t)^{-\beta}, \quad \zeta^0(t):=\zeta^0(T-t)^{-\beta},
$$
let us remark that \eqref{interm2} entails that $\underline{U}$ is actually a true weak solution to Eq. \eqref{eq1} except at $r=|x|=\zeta_0(t)$, $t\in(0,T)$, as the contact condition $(\underline{U}^m)_r(r,t)=0$ (see \cite[Section 9.8]{VazquezPME}) fails there. If the inequality $R(t)\leq\zeta_0(t)$ (valid at $t=0$ by \eqref{interm3}) holds true for any $t\in(0,T)$, then the proposition is proved. In the opposite situation, there is $t_1\in(0,T)$ such that the solution $u(t_1)$ and the subsolution $\underline{U}(t_1)$ have a single contact point in radially symmetric variables $r_0=|x_0|>0$ such that
$$
u(r_0,t_1)=\underline{U}(r_0,t_1), \quad \left\{\begin{array}{ll}u(x,t_1)>\underline{U}(x,t_1), & {\rm if} \ \zeta_0(t_1)\leq |x|<r_0,\\ u(x,t_1)<\underline{U}(x,t_1), & {\rm if} \ r_0<|x|<R(t_1),\end{array}\right.
$$
and in particular $R(t_1)<\zeta^0(t_1)$. Since both functions have compact support, we conclude by applying the intersection comparison technique. Indeed, in dimension $N=1$, an application of \cite[Proposition 2, p. 242]{S4} and \cite[Proposition 3, p. 243]{S4} (the latter on the interval $\zeta_0(t),\zeta^0(t)$) ensures that, for any $t\in(t_1,T)$, the profiles of $u(t)$ and $\underline{U}(t)$ will have exactly a single intersection point, lying in the interval $(\zeta_0(t),\zeta^0(t))$, which readily implies that $R(t)\leq\zeta^0(t)$ for any $t\in(t_1,T)$. The construction ensures that, in fact, $R(t)\leq\zeta^0(t)$ for any $t\in(0,T)$ and the proof of \eqref{uppersup} is complete, with $C_0=\zeta^0$. The theory of intersection comparison employed above works also in dimension $N\geq2$ between radially symmetric functions (see for example \cite[Section 5]{GV94} and references therein or \cite[Section 6.3]{V22}) and thus the previous proof still works in radial variables, as claimed.
\end{proof}
This proposition, which is also interesting by itself, allows us to establish the blow-up rates stated in Theorem \ref{th.1}.
\begin{proof}[Proof of Theorem \ref{th.1}]
The proof is divided in two steps: in the first one we establish the lower blow-up rate, while the second and more involved deals with the upper blow-up rate.

\medskip

\noindent \textbf{Lower blow-up rate.} While for $\sigma=0$ this rate is absolutely straightforward by comparison with the solution to the ODE $u_t=u^p$, in our case we have to work more and make strong use of the estimate \eqref{uppersup}. For $t\in(0,T)$, let us define
$$
g(t)=\|u(\cdot,t)\|_{\infty}=u(r(t),t).
$$
Since $r(t)$ is a point where $u$ reaches the maximum at time $t$, we have that for $0<\tau<t<T$
$$
u(r(\tau),t)-u(r(\tau),\tau)\le g(t)-g(\tau)\le u(r(t),t)-u(r(t),\tau).
$$
We then deduce from the Mean Value Theorem that
    \begin{equation*}
        u_t(r(\tau),\tau_2) \leq \frac{g(t)-g(\tau)}{t-\tau} \leq u_t(r(t),\tau_1),
    \end{equation*}
for some $\tau_1,\tau_2\in[\tau,t]$. Since $u_t$ is locally bounded, we get that $g(t)$ is locally Lipschitz, therefore $g$ is differentiable for almost every time and, by passing to the limit as $\tau\to t$ in the previous inequalities, we find:
    $$
g'(t):=\partial_t(u(r(t),t)) = u_t(r(t),t).
    $$
We then infer from Eq. \eqref{eq1} written in radially symmetric variables that
$$
u_t(r(t),t)=(u^m)_{rr}(r(t),t)+\frac{N-1}{r(t)}(u^m)_{r}(r(t),t)+r(t)^{\sigma}u^p(r(t),t)\leq C_0^{\sigma}(T-t)^{-\beta\sigma}u^p(r(t),t),
$$
since $r(t)\leq R(t)\leq C_0(T-t)^{-\beta}$ by Proposition \ref{prop.support} and $(u^m)_r(r(t),t)=0$, $(u^m)_{rr}(r(t),t)\leq0$ following from the fact that $r(t)$ is a maximum point for $u(t)$. Therefore, we obtain the differential inequality
$$
g'(t)\leq C_0^{\sigma}(T-t)^{-\beta\sigma}g(t)^p,
$$
or equivalently, taking into account that $p>1$,
$$
\frac{d}{dt}g(t)^{1-p}=(1-p)g(t)^{-p}g'(t)\geq-C_0^{\sigma}(p-1)(T-t)^{-\beta\sigma}.
$$
We integrate the previous inequality on $(t,T)$ for generic $t\in(0,T)$ and, taking into account that
$$
1-\beta\sigma=\frac{(\sigma+2)(p-1)}{L}=(p-1)\alpha>0,
$$
we find
\begin{equation*}
\begin{split}
-g(t)^{1-p}&=g(T)^{1-p}-g(t)^{1-p}\geq-C_0^{\sigma}(p-1)\int_t^{T}(T-s)^{-\beta\sigma}\,ds\\
&=-\frac{C_0^{\sigma}(p-1)}{1-\beta\sigma}(T-t)^{1-\beta\sigma},
\end{split}
\end{equation*}
hence
$$
g(t)\geq\left[\frac{C_0^{\sigma}(p-1)}{1-\beta\sigma}\right]^{-1/(p-1)}(T-t)^{-(1-\beta\sigma)/(p-1)}=C_1(T-t)^{-\alpha},
$$
which is the claimed lower bound.

\medskip

\noindent \textbf{Upper blow-up rate.} We proceed by adapting a technique stemming from \cite{Hu96, FS01} (see also \cite{FdP21}), employing both Proposition \ref{prop.support} and the lower blow-up rate established above to estimate the new terms coming from the space-dependent coefficient. Set
$$
M(t):=\sup\{u(x,\tau): x\in\real^N, \tau\in(0,t)\}, \quad t\in(0,T).
$$
It is obvious that $M(t)$ is a continuous and non-decreasing function with $M(t)\to\infty$ as $t\to T$. Fix $t_0\in(0,T)$ and define recursively the following sequence:
\begin{equation}\label{seq}
t_{j+1}:=\sup\{t\in(t_j,T): M(t)=2M(t_j)\}, \quad j\geq0.
\end{equation}
We notice that \eqref{seq} implies that $\|u(t_j)\|_{\infty}=M(t_j)$ for any $j\geq0$, hence there exists $x_j\in\real^N$ such that $u(x_j,t_j)=M(t_j)$. Assume now for contradiction that there exists a subsequence (not relabeled) $t_j$ such that
\begin{equation}\label{interm4}
\lim\limits_{j\to\infty}(t_{j+1}-t_j)M(t_j)^{1/\alpha}=\infty,
\end{equation}
with $\alpha$ defined in \eqref{SSexp}. Indeed, if there is no subsequence satisfying \eqref{interm4}, then we would find that there exists $K>0$ such that
\begin{equation}\label{interm5}
t_{j+1}-t_j\leq KM(t_j)^{-1/\alpha}, \quad j\geq0,
\end{equation}
and by adding up the inequalities \eqref{interm5} for every $j$ and taking into account that $t_j\to T$ as $j\to\infty$, we obtain
$$
T-t_0\leq K\sum\limits_{j=0}^{\infty}M(t_j)^{-1/\alpha}=KM(t_0)^{-1/\alpha}\sum\limits_{j=0}^{\infty}2^{-j/\alpha}\leq\overline{K}M(t_0)^{-1/\alpha},
$$
and thus
$$
M(t_0)\leq\left(\frac{T-t_0}{\overline{K}}\right)^{-\alpha}, \quad t_0\in(0,T),
$$
and the arbitrary choice of $t_0\in(0,T)$ implies the desired upper blow-up rate.

Let us thus come back to the assumption \eqref{interm4} and derive a contradiction from it. To this end, we define the following rescaling:
\begin{equation}\label{interm6}
\begin{split}
\varphi_j(y,s)&:=\frac{1}{M(t_j)}u\big(M(t_j)^{\beta/\alpha}y+x_j,M(t_j)^{-1/\alpha}s+t_j\big),\\ &(y,s)\in\real^N\times(-t_jM(t_j)^{1/\alpha},(T-t_j)M(t_j)^{1/\alpha}),
\end{split}
\end{equation}
for any $j\geq0$. Straightforward calculations starting from the definition \eqref{interm6} lead to the following identities:
$$
\partial_s\varphi_j(y,s)=M(t_j)^{-1/\alpha-1}\partial_tu\big(M(t_j)^{\beta/\alpha}y+x_j,M(t_j)^{-1/\alpha}s+t_j\big),
$$
$$
\Delta\varphi_j^m(y,s)=M(t_j)^{2\beta/\alpha-m}\Delta u^m\big(M(t_j)^{\beta/\alpha}y+x_j,M(t_j)^{-1/\alpha}s+t_j\big),
$$
and
$$
\big|M(t_j)^{\beta/\alpha}y+x_j\big|^{\sigma}u^p\big(M(t_j)^{\beta/\alpha}y+x_j,M(t_j)^{-1/\alpha}s+t_j\big)=M(t_j)^{p+\sigma\beta/\alpha}\left|y+\frac{x_j}{M(t_j)^{\beta/\alpha}}\right|^{\sigma}\varphi_{j}^p(y,s).
$$
Noticing that
$$
\frac{1}{\alpha}+1=m-\frac{2\beta}{\alpha}=p+\frac{\sigma\beta}{\alpha}=\frac{m\sigma+2p}{\sigma+2},
$$
we readily conclude from the previous identities and performing obvious simplifications that $\varphi_j$ is a solution to the equation
\begin{equation}\label{interm7}
\partial_s\varphi_j=\Delta\varphi_j^m+\left|y+\frac{x_j}{M(t_j)^{\beta/\alpha}}\right|^{\sigma}\varphi_j^p, \quad (y,s)\in\real^N\times(-t_jM(t_j)^{1/\alpha},(T-t_j)M(t_j)^{1/\alpha}),
\end{equation}
satisfying moreover $\varphi_j(0,0)=1$ and
\begin{equation}\label{interm8}
0\leq\varphi_j(y,s)\leq2, \quad (y,s)\in\real^N\times(-t_jM(t_j)^{1/\alpha},(t_{j+1}-t_j)M(t_j)^{1/\alpha}),
\end{equation}
for any $j\geq0$. We next infer from Proposition \ref{prop.support} and the lower bound that
$$
|x_j|\leq C_0(T-t_j)^{-\beta}, \quad M(t_j)\geq C_1(T-t_j)^{-\alpha},
$$
hence
$$
\frac{x_j}{M(t_j)^{\beta/\alpha}}\leq\frac{C_0}{C_1^{\beta/\alpha}}=:C_3<\infty.
$$
This fact, together with the bound \eqref{interm8}, the assumption \eqref{interm4} and the uniform interior Schauder's estimates for $\varphi_j$ (see for example \cite{Lieberman}), imply the existence of $C_4\in(0,C_3]$ and a subsequence (not relabeled) $(\varphi_j)_{j\geq1}$ converging in $C^{2+\sigma,1+\sigma/2}_{\rm loc}(\real^N\times(-\infty,\infty))$ to a solution $\varphi$ to the equation (obtained by passing to the limit in \eqref{interm7})
\begin{equation}\label{interm9}
\partial_s\varphi=\Delta\varphi^m+|y+C_4|^{\sigma}\varphi^p, \quad (y,s)\in\real^N\times\real,
\end{equation}
such that $\varphi(0,0)=1$ and $0\leq\varphi(y,s)\leq2$ for $(y,s)\in\real^N\times(0,\infty)$, that is, a non-trivial global solution. But this is a contradiction, since an easy translation argument together with the outcome of \cite[Section 5]{ILS24b} prove that any non-negative and non-trivial solution to \eqref{interm9} blows up in finite time if $p\in(1,m)$. This contradiction proves that \eqref{interm4} cannot hold true and the proof is complete.
\end{proof}

\section{Blow-up points and sets}\label{sec.points}

This section is dedicated to some results related to blow-up points and sets, making precise the discussion given in the Introduction. In the next lines, $\overline{D}$ and $\partial D$ will denote the closure, respectively the boundary of a set $D\subset\real^N$. The first result shows that uniform boundedness on the boundary of a ball $\partial B(0,r_0)$ implies uniform boundedness in the closed ball $\overline{B(0,r_0)}$.
\begin{proposition}\label{prop.bounded}
Let $m$, $p$ and $\sigma$ be as in \eqref{range.exp} and $u$ be a solution to the Cauchy problem \eqref{eq1}-\eqref{icond}. Let $T\in(0,\infty)$ be its blow-up time and let $r_0>0$ be such that $u(x,t)$ is uniformly bounded for any $x\in\real^N$ with $|x|=r_0$ and for $t\in(0,T)$. Then $\overline{B(0,r_0)}\cap B(u_0)=\emptyset$, where $B(u_0)$ is the blow-up set of $u$.
\end{proposition}
\begin{proof}
Fix $M>0$ sufficiently large such that
$$
M>\sup\{u_0(x): x\in\overline{B(0,r_0)}\}, \quad M>\sup\{u(x,t): x\in\partial B(0,r_0), t\in(0,T)\}.
$$
The existence of $M$ follows from the uniform boundedness assumption for $x\in\partial B(0,r_0)$ and $t\in(0,T)$. Consider now the Cauchy-Dirichlet problem
\begin{equation}\label{prob1}
\left\{\begin{array}{ll}v_t=\Delta v^m+r_0^{\sigma}v^p, & (x,t)\in B(0,r_0)\times(0,T), \\
u(x,t)=M, & (x,t)\in\partial B(0,r_0)\times(0,T), \\ u(x,0)=M, & x\in B(0,r_0).\end{array}\right.
\end{equation}
The choice of $M$ ensures that any solution to the problem \eqref{prob1} is a supersolution to the Cauchy problem \eqref{eq1}-\eqref{icond} on $B(0,r_0)\times(0,T)$. Consider next the solution to the elliptic Dirichlet problem
\begin{equation}\label{prob2}
\left\{\begin{array}{ll}\Delta w^m+r_0^{\sigma}w^p=0, & x\in B(0,r_0),\\
w(x)=M, & x\in\partial B(0,r_0).\end{array}\right.
\end{equation}
whose existence and radial symmetry when $p<m$ are ensured by classical results (see \cite{BO86, GNN79}) together with an easy scaling argument. Since $$
\frac{d}{dr}\left[r^{N-1}\frac{dw^m}{dr}\right]=-r_0^{\sigma+N-1}w^p<0, \quad r\in(0,r_0),
$$
it follows that $r\mapsto r^{N-1}(w^m)_r$ is a decreasing function on $(0,r_0)$. In particular, the radial symmetry entails that $(w^m)_r(0)=0$ and thus $r^{N-1}(w^m)_r\leq0$ for $r\in(0,r_0)$. We further infer that $w^m$ is a decreasing function with respect to the radial variable and thus $w^m(x)\geq M^m$ or, equivalently, $w(x)\geq M$ for any $x\in\overline{B(0,r_0)}$. We have just established that the (stationary) solution $w$ to \eqref{prob2} is a supersolution to the problem \eqref{prob1}. We infer from the comparison principle for the Cauchy-Dirichlet problem (see for example \cite[Proposition 2.2]{Su02}) that the solution to the problem \eqref{prob1} is global in time and, thus, our initial solution $u$ is also uniformly bounded on $\overline{B(0,r_0)}$, as claimed.
\end{proof}
We follow with a partial result completing the panorama of blow-up sets. It shows that, if a solution has a blow-up point with a sufficiently fast pointwise blow-up rate, then it blows up at any point in $\real^N$.
\begin{proposition}\label{prop.unbounded}
Let $m$, $p$ and $\sigma$ be as in \eqref{range.exp} and $u$ be a (radially symmetric) solution to the Cauchy problem \eqref{eq1}-\eqref{icond} with blow-up time $T(u_0)\in(0,\infty)$. Let $r_0>0$ be such that $u$ blows up at time $t=T(u_0)$ at any point $x\in\real^N$ with $|x|=r_0$ with
\begin{equation}\label{interm19}
\lim\limits_{t\to T(u_0)}(T(u_0)-t)^{1/(m-1)}u(r_0,t)\geq C_0,
\end{equation}
for some $C_0>0$ sufficiently large (depending on $m$ and $r_0$). Then $B(u)=\real^N$.
\end{proposition}
\begin{proof}
Let $r_0>0$ be as in the statement. Then, taking into account the radial symmetry of $u$, there is $t(C_0)\in(0,T(u_0))$ such that
$$
u(x,t)>\frac{C_0}{2}(T-t)^{-1/(m-1)}, \quad x\in\real^N, \quad |x|=r_0, \quad t\in(t(C_0),T(u_0)).
$$
Since $u$ is a solution to \eqref{eq1}, it is thus a supersolution to the standard porous medium equation $u_t=\Delta u^m$ and we derive from \cite[Theorem 1.2]{V09} and the comparison principle that, if $C_0$ is sufficiently large, $B(0,r_0)\subset B(u)$. Indeed, an argument completely analogous to the proof of Proposition \ref{prop.bounded} based on comparison with stationary supersolutions shows that $u$ cannot blow up at a prior time $T<T(u_0)$ while the boundary condition on $\partial B(0,r_0)$ is still finite, thus ensuring (together with \cite[Theorem 1.2]{V09}) that blow up takes place at any point of $B(0,r_0)$ at the same time $T(u_0)$. Assume then for contradiction that $B(u)\neq\real^N$, thus, there is $r_1>r_0$ such that $u(x,t)$ is uniformly bounded for any $x\in\real^N$ such that $|x|=r_1$ and for $t\in(0,T(u_0))$. But then Proposition \ref{prop.bounded} implies that $u(x,t)$ is uniformly bounded for $(x,t)\in B(0,r_1)\times(0,T(u_0))$, contradicting the fact that $B(0,r_0)\subset B(u)$. Thus, $B(u)=\real^N$, as claimed.
\end{proof}

The previous results shows that, under the restriction \eqref{interm19}, the only possible blow-up sets of a solution $u$ to Eq. \eqref{eq1} are either $B(u)=\real^N$ or $B(u)=\{\infty\}$. We strongly believe that this dichotomy is true for any solution $u$ to Eq. \eqref{eq1}, and the condition \eqref{interm19} is only technical, but new ideas are to be found in order to remove it. We also expect that the outcome of Proposition \ref{prop.unbounded} might be useful in further developments related to the open problems discussed at the end of the Introduction.

However, let us stress here that the \emph{limiting behavior} given in \eqref{interm19} is actually \emph{attained on self-similar solutions} in the form \eqref{SSS} to Eq. \eqref{eq1} presenting either compact support or a tail as $|x|\to\infty$ and a profile with local behavior near the origin given by
\begin{equation*}
\lim\limits_{\xi\to0}\xi^{-2/(m-1)}f(\xi)=\left[\frac{m-1}{2m(mN-N+2)}\right]^{1/(m-1)},
\end{equation*}
which exist (at least for some critical value of $\sigma$) according to \cite[Theorem 1.3]{ILS24a} (dimension $N\geq2$) and \cite[Theorem 1.4]{IS21} (dimension $N=1$). It is shown in the quoted references that the corresponding self-similar solutions blow up at time $T$ with the same rate of the $L^{\infty}$ norm as in Theorem \ref{th.1}, with blow-up set $B(u)=\real^N$, but with a pointwise blow-up rate at a fixed $x\in\real^N$ achieving equality in \eqref{interm19}.

\medskip

\noindent \textbf{Discussion about blow-up sets.} Let us first recall that, as proved in \cite{IS21, ILS24a}, there exist $\sigma_0$, $\sigma_1$ such that $0<\sigma_0<\sigma_1<\infty$ and:

$\bullet$ for any $\sigma\in(0,\sigma_0)$, there exists $A(\sigma)>0$ such that the self-similar profiles $f(\cdot;A)$ solutions to \eqref{SSODE} and subject to the initial conditions
\begin{equation}\label{initcond.prof}
f(0;A)=A\in(0,A(\sigma)), \quad f'(0;A)=0,
\end{equation}
are strictly positive for any $\xi>0$. In dimension $N=1$, the existence of such a value of $\sigma_0$ is ensured by an inspection of the proof of \cite[Theorem 1.5]{IS21} together with a standard argument based on the behavior near the saddle point denoted by $P_2$ in the dynamical system in the above mentioned proof, and in fact, it follows by this argument and the statement of \cite[Theorem 1.5]{IS21} that profiles satisfying \eqref{initcond.prof} have a tail behavior given by
\begin{equation}\label{tail}
\lim\limits_{\xi\to\infty}\xi^{\sigma/(p-1)}f(\xi;A)=\left(\frac{1}{p-1}\right)^{1/(p-1)}.
\end{equation}
In dimensions $N\geq2$, the existence of $\sigma_0$ for which profiles satisfying \eqref{initcond.prof} are strictly positive is granted by the results of \cite[Section 5]{ILS24a} together with an argument of continuity with respect to $\sigma$ in a neighborhood of $\sigma=0$, although we did not prove the expected tail behavior \eqref{tail} in this case.

$\bullet$ for any $\sigma\in(\sigma_1,\infty)$, there exists $K\in(0,\infty)$ (depending on $\sigma$) and compactly supported self-similar profiles $f_{K}$ such that
\begin{equation}\label{interm10}
\lim\limits_{\xi\to0}\frac{f_{K}(\xi)}{\xi^{(\sigma+2)/(m-p)}}=K.
\end{equation}
This existence result stems from \cite[Theorem 1.4]{IS21} in dimension $N=1$ and \cite[Theorem 1.3]{ILS24a} in dimension $N\geq2$.

Let $m$, $p$ as in \eqref{range.exp} and let $\sigma_0$ and $\sigma_1$ as above. Let $u_0$ be an initial condition as in \eqref{icond} and $T(u_0)$ be the blow-up time of the solution $u$ to Eq. \eqref{eq1} with initial condition $u_0$. We next discuss the following cases:

\medskip

\textbf{(a)} For $\sigma\in(0,\sigma_0)$, assume that $u_0(0)>0$ and there is $A<\min\{u_0(0),A(\sigma)\}$ such that $T(u_0)^{\alpha}u_0(T(u_0)^{-\beta}\xi)$ and $f(\xi;A)$ have a single point of intersection. Define the self-similar solution
\begin{equation}\label{SSSA}
U(x,t;A)=(T(u_0)-t)^{-\alpha}f(|x|(T(u_0)-t)^{\beta};A)
\end{equation}
The latter property gives that the radially symmetric profiles of $u_0$ and $U(x,0;A)$ have a single intersection point. The fact that $u$ and $U(\cdot,\cdot;A)$ have the same blow-up time and an intersection comparison argument (see \cite[Propositions 1 and 3, pp. 240-243]{S4} in dimension $N=1$ and \cite[Section 5]{GV94} or \cite[Section 6.3]{V22} for a generalization of the theory of intersection comparison in radially symmetric variables in dimension $N\geq2$) imply that the profiles in radially symmetric variables of $u(t)$ and $U(\cdot,t;A)$ have a single point of intersection for any $t\in(0,T(u_0))$. Moreover, since $u_0$ is compactly supported, it follows by the finite speed of propagation \cite[Theorem 1.5]{ILS24b} that $u(t)$ is compactly supported for any $t\in(0,T(u_0))$. On the contrary, the self-similar solution \eqref{SSSA} is positive for any $t\in(0,T(u_0))$. This ordering and the uniqueness of the intersection point readily give that
$$
u(0,t)\geq U(0,t;A), \quad t\in(0,T(u_0)).
$$
Since the solution given in \eqref{SSSA} blows up at $x=0$, it follows that the origin is a blow-up point for $u$. Proposition \ref{prop.bounded} ensures then that $B(u)=\real^N$.

\medskip

\textbf{(b)} For $\sigma>\sigma_1$, assume that the initial condition $u_0$ satisfies $u_0\equiv0$ on $B(0,\delta)$ for some $\delta>0$ and that there is $K\in(0,\infty)$ and a profile $f_K$ as in \eqref{interm10} such that $T(u_0)^{\alpha}u_0(T(u_0)^{-\beta}\xi)$ and $f_K$ have a single point of intersection. Introduce the self-similar solution
\begin{equation*}
U_K(x,t)=(T(u_0)-t)^{-\alpha}f_K(|x|(T(u_0)-t)^{\beta}).
\end{equation*}
Thus, $u_0$ and $U_K(x,0)$ have a single point of intersection, and a similar argument as in part (a) entails that $u(t)$ and $U_K(t)$ have a single point of intersection (in radial variables), for any $t\in(0,T(u_0))$. This means that the initial ordering of the two functions remains unchanged for any $t\in(0,T(u_0))$; that is, for any $t\in(0,T)$, the left edge of the support of $u(t)$ lies in the positivity set of $U_K(t)$, while the right edge of the support of $u(t)$ lies in the zero set of $U_K(t)$. Since $U_K$ blows up only at infinity, it is expected that the same property will happen for $u$, but a proof of this fact requires the removal of the condition \eqref{interm19} from Proposition \ref{prop.unbounded}.

\bigskip

\noindent \textbf{Acknowledgements} R. G. I. and A. S. are partially supported by the Spanish project PID2024-160967NB-I00 funded by Agencia Estatal de Investigaci\'on (AEI) of Spain.

\bigskip

\noindent \textbf{Data availability} Our manuscript has no associated data.

\bigskip

\noindent \textbf{Conflict of interest} The authors declare that there is no conflict of interest.

\bibliographystyle{plain}

\end{document}